\documentclass{amsart}

\numberwithin{equation}{section}

\usepackage[centertags]{amsmath}
\usepackage{amsfonts}
\usepackage{amssymb}
\usepackage{amsthm}
\usepackage{amscd}
\usepackage{hyperref}
\usepackage{algorithmic, algorithm}
\usepackage[all]{xy}

\usepackage{mathrsfs} 

\DeclareFontFamily{OT1}{rsfs}{}
\DeclareFontShape{OT1}{rsfs}{n}{it}{<-> rsfs10}{}
\DeclareMathAlphabet{\mathscr}{OT1}{rsfs}{n}{it}

\newcommand{\cO}{\mathcal{O}}

\newcommand{\cL}{\mathcal{L}}

\newcommand{\ds}{\displaystyle}

\newcommand{\ra}{\rightarrow}

\newcommand{\eps}{\varepsilon}
\newcommand{\vphi}{\varphi}

\newcommand{\comment}[1]{}
\newcommand{\Q}{\mathbf{Q}}

\newcommand{\R}{\mathbf{R}}

\newcommand{\cS}{\mathcal{S}}

\newcommand{\C}{\mathbf{C}}

\newcommand{\cT}{\mathcal{T}}

\newcommand{\Z}{\mathbf{Z}}
\newcommand{\F}{\mathbf{F}}

\newcommand{\Fbar}{\overline{\F}}

\newcommand{\isom}{\cong}

\DeclareFontEncoding{OT2}{}{} 

\renewcommand{\H}{\HH}

\DeclareMathOperator{\Tr}{Tr}

\DeclareMathOperator{\ord}{ord}

\DeclareMathOperator{\Gal}{Gal}

\DeclareMathOperator{\HH}{H}

\DeclareMathOperator{\Hom}{Hom}

\DeclareMathOperator{\End}{End}

\DeclareMathOperator{\Jac}{Jac}
\DeclareMathOperator{\Cl}{Cl}

\theoremstyle{plain} 
\newtheorem{thm}{Theorem}[section] 

\newtheorem{cor}[thm]{Corollary}
\newtheorem{lem}[thm]{Lemma}

\theoremstyle{definition} 
\newtheorem{defn}[thm]{Definition} 
\theoremstyle{remark} 
\newtheorem{rem}{Remark}
\newtheorem{example}{Example}

\newcounter{tasknumber}
\newcommand{\task}[2][]{%
  \addtocounter{tasknumber}{1}%
  \begin{center}%
  \framebox[1.1\width]{\begin{minipage}{0.9\textwidth}%
  \textbf{Task \arabic{tasknumber}} \textit{\if!#1(unassigned)!\else (#1)\fi}: {#2}%
  \end{minipage}}%
  \end{center}%
}

\newcounter{assumptionnumber}
\newcommand{\assumption}[2][]{%
  \addtocounter{assumptionnumber}{1}%
  \begin{center}%
  \framebox[1.1\width]{\begin{minipage}{0.9\textwidth}%
  \textbf{Assumption \arabic{assumptionnumber}} \textit{\if!#1!\else (#1)\fi}: {#2}%
  \end{minipage}}%
  \end{center}%
}

\newcommand*{\Cay}{\mathrm{Cay}}
\newcommand*{\frakl}{\mathfrak{l}}

\begin{document}

\title[Horizontal isogeny graphs of ordinary abelian varieties]{Horizontal isogeny graphs of ordinary abelian varieties and the discrete logarithm problem}


\author{Dimitar Jetchev}
\email{dimitar.jetchev@epfl.ch}
\address{Ecole Polytechnique F\'ed\'erale de Lausanne, EPFL SB GR-JET, Switzerland}


\author{Benjamin Wesolowski}
\email{benjamin.wesolowski@epfl.ch}
\address{Ecole Polytechnique F\'ed\'erale de Lausanne, , EPFL IC LACAL, Switzerland}

\date{\today}


\begin{abstract}
Fix an ordinary abelian variety defined over a finite field. The ideal class group of its endomorphism ring acts freely on the set of isogenous varieties with same endomorphism ring, by complex multiplication. Any subgroup of the class group, and generating set thereof, induces an isogeny graph on the orbit of the variety for this subgroup. We compute (under the Generalized Riemann Hypothesis) some bounds on the norms of prime ideals generating it, such that the associated graph has good expansion properties.

We use these graphs, together with a recent algorithm of Dudeanu, Jetchev and Robert for computing explicit isogenies in genus 2, to prove random self-reducibility of the discrete logarithm problem within the subclasses of principally polarizable ordinary abelian surfaces with fixed endomorphism ring. In addition, we remove the heuristics in the complexity analysis of an algorithm of Galbraith for explicitly computing isogenies between two elliptic curves in the same isogeny class, and extend it to a more general setting including genus~2.
\end{abstract}

\maketitle

\section{Introduction}

\subsection{Motivation}\label{subsec:motivation}
Let $\mathscr C$ be a hyperelliptic curve of genus $g$ over defined over a finite field $\mathbf F_q$ and let $\mathscr J = \Jac(\mathscr C)$ be its Jacobian  
-- a principally polarized abelian surface over $\F_q$. The discrete logarithm problem (or DLP) in genus $g$ is the following: given $P \in \mathscr J(\F_q)$ and $Q = rP \in \mathscr J(\F_q)$ for some secret multiplier $r$, compute~$r$. The problem for $g = 1$ is known as the elliptic curve discrete logarithm problem (or ECDLP); it is a central tool in public key cryptography, and has been extensively studied since its introduction in the 1980's~\cite{Miller1986,koblitz1987elliptic}. The case of $g=2$ has been shown to be a promising alternative, allowing very efficient arithmetic~\cite{gaudry:fastgen2,Bos2016}, but very little is known about the hardness of the corresponding version of the DLP. Apart from the question of the hardness of the problem on a particular Jacobian, one may ask how the difficulty of the problem compares on two distinct Jacobians.
A natural way of transferring the problem from one Jacobian to another is via isogenies.
It is thus of interest to study whether two Jacobians of genus 2 curves have the same difficulty of the problem, assuming that there exists an isogeny between them. 
Tate's isogeny theorem~\cite{tate:isogeny} implies that two abelian surfaces over a finite field are isogenous if and only if the characteristic polynomials of the Frobenius acting on their $\ell$-adic Tate modules are the same.
The latter can be computed efficiently, so it is easy to determine if two Jacobians are isogenous. It is however not clear how to explicitly compute an isogeny between two such Jacobians, which is actually needed to transfer the discrete logarithm problem. 

The case of ordinary elliptic curves has been treated by Jao, Miller and Venkatesan \cite{jao-miller-venkatesan,JMV09} using random walks on isogeny graphs and rapid mixing arguments.
A crucial ingredient in their analysis is that one can efficiently compute isogenies of small degrees, polylogarithmic in $q$. More precisely, one considers a graph with vertices the set of isomorphism classes of elliptic curves in the isogeny class that have a fixed endomorphism ring. These isomorphism classes correspond, by CM theory, to the ideal classes of that endomorphism ring. The edges of the graph correspond to \emph{horizontal} isogenies, that is, $\mathfrak a$-transforms in the language of \cite{taniyama-shimura}. It turns out that it is connected for suitably chosen bounds on the ideal norms and, under GRH, it rapidly mixes random walks (i.e., behaves as an expander graph). Via random walks on this graph, one can show that it is possible to reduce the discrete logarithm problem from a given curve to the problem to a uniformly random curve in that class, thus obtaining random self-reducibility of the elliptic curve discrete logarithm problem within the class.  

The similar problem in genus 2 is much more challenging since, unlike elliptic curves, abelian surfaces are not \emph{a priori} principally polarized, so a quotient of a Jacobian by a finite subgroup need not be the Jacobian of a curve.
Even if it is, there might be multiple non-equivalent principal polarizations giving rise to non-isomorphic curves\footnote{An example of this phenomenon has been given by Howe \cite{howe:distcurves}. More precisely, Howe showed that the curves $y^2=x^5+x^3+x^2 - x - 1$ and $y^2=x^5 - x^3+x^2 - x - 1$ over $\F_{11}$ are not isomorphic; yet, their Jacobians are absolutely simple and isomorphic as non-polarized abelian surfaces.}.
In addition, if one tries the straightforward analogy to~\cite{jao-miller-venkatesan} of constructing isogeny graphs with vertices that are ideal classes in the class group of the endomorphism algebra (in this case, a quartic CM-field), one may get abelian surfaces that are not even principally polarizable and hence, unsuitable for transferring the discrete logarithm problem in practice. 
Finally, even if the target is principally polarizable, for the purpose of proving random self-reducibility, one does not need just one principal polarization on the target, but all of them, or at least the capability to sample one uniformly at random. 

\subsection{Main theorem}\label{subsec:mainthm}

Jacobians of genus 2 hyperelliptic curves will be seen as a particular case of the following, more general situation. Let $\mathscr A$ be an absolutely simple, ordinary abelian variety of dimension $g$ over a finite field and let 
$K = \End(\mathscr A) \otimes \Q$ be the corresponding CM field. The endomorphism ring $\End(\mathscr A)$ 
is isomorphic to an order $\cO$ of conductor $\mathfrak f$ in $K$. The ideal class group $\Cl(\cO)$ acts freely on the set of varieties isogenous to $\mathscr A$ with same endomorphism ring $\cO$, by complex multiplication. Let $H \subseteq \Cl(\cO)$ be any subgroup and let $H(\mathscr A)$ the $H$-orbit of $\mathscr A$. The choice of a set $\mathcal S$ of invertible ideals in $\cO$ generating $H$ induces a graph whose set of vertices is $H(\mathscr A)$ and whose edges are labelled with isogenies between these abelian varieties. The norms of the ideals in $\mathcal S$ are exactly the degrees of the induced isogenies. For any $B>0$ and ideal $\mathfrak m$ in $\cO$, let $\mathcal S_B$ be the set of ideals in $\cO$ of prime norm and coprime to~$\mathfrak f\mathfrak m$. Let $\mathscr G_B$ be the induced isogeny graph, where all the degrees are bounded by~$B$.

\begin{thm}[Rapid mixing for $H(\mathscr A)$]\label{thm:rapidmixgeneral}
Assuming the Generalized Riemann Hypothesis, for any $\eps > 0$, there exists a bound 
$$
B = O \left(\left(g[\Cl(\cO) : H]\ln (d_KN(\mathfrak f\mathfrak m))\right)^{2 + \varepsilon}\right), 
$$ 
such that for any subset $W$
of $H(\mathscr A)$, any random walk in $\mathscr G_B$ of length at least $\ln(2|H|/|W|^{1/2})$ starting from a given vertex will end in $W$ with probability between $|W|/(2|H|)$ and $3|W|/(2|H|)$. In particular, the regular graph $\mathscr G_B$ is connected and rapidly mixes random walks. 
\end{thm}

It is worth noticing that even the connectivity of the graph is new: the classical bounds for connectivity are derived from Bach's bounds~\cite{Ba90}, which can only be applied when $H$ is the full class group $\Cl(\cO)$.
We will prove Theorem~\ref{subsec:mainthm} as a corollary of the following theorem. It constructs and proves that certain Cayley graphs for subgroups of more general ray class groups are expanders.

\begin{thm}\label{thm:eigenCharacGeneral}
Let $K$ be a number field of degree $n$ and discriminant $d_K$, $\mathfrak m$ an integral ideal of $\cO_K$, $G$ the narrow ray class group of $K$ modulo $\mathfrak m$, and $H$ a subgroup of $G$. For any ideal $\mathfrak l$ of $\cO_K$ coprime to $\mathfrak m$, let $[\mathfrak l]$ denote its image in $G$. Let 
$$
\cT_{H, \mathfrak m}(B) = \{\text{prime ideals }\mathfrak l \text{ of }\cO_K \mid (\mathfrak l, \mathfrak m) = 1, N\mathfrak l \leq B \text{ is prime}\text{ and } [\mathfrak l] \in H\}.
$$
Let $T_{H, \mathfrak m} (B)$ be the multiset of its image in $G$. Let $\mathscr G_B$ be the graph whose vertices are the elements of $H$ and whose non-oriented edges are precisely $(h, sh)$ for any $h \in H$ and $s \in T_{H, \mathfrak m}(B)$.
Assuming the Generalized Riemann Hypothesis, for any character $\chi$ of $H$, the corresponding eigenvalue $\lambda_\chi$ of the Cayley graph $\mathscr G_B$ satisfies
$$
\lambda_{\chi} = \frac{\delta(\chi)}{[G:H]}\mathrm{li}(B) + O\left(nB^{1/2}\ln(Bd_KN\mathfrak m)\right),
$$
where $\delta(\chi)$ is $1$ if $\chi$ is trivial, and $0$ otherwise. The implied constants are absolute.
\end{thm}

Note that a similar result is proven in~\cite{JMV09}, where $H$ is the full narrow ray class group, rather than a subgroup.
It was sufficient to study isogeny graphs of elliptic curves, which can be represented as Cayley graphs of class groups in imaginary quadratic fields. However, it is not strong enough for higher genus, where one needs to work on subgroups of class group of CM-fields, due to the extra condition of principal polarizability. Since properties of expander graphs do not transfer nicely to subgraphs in general, the refinement provided by Theorem~\ref{thm:eigenCharacGeneral} is crucial.

\subsection{Applications of Theorem~\ref{thm:rapidmixgeneral}}\label{subsec:apps}

Using the CM theory for polarized class groups, we will apply Theorem~\ref{thm:rapidmixgeneral} to analyse isogeny graphs of Jacobians of hyperelliptic curves of genus 2.
More precisely let $\mathscr A$ be an absolutely simple, ordinary principally polarizable abelian surface. Let $K$ be its quartic CM-field, and let $K_0$ be its real quadratic subfield, and let $\cO$ be the order in $K$ isomorphic to $\End(\mathscr A)$. Let $\mathscr P(\cO) \unlhd \Cl(\cO)$ be the image of the natural projection of the Shimura class group $\mathfrak C(\cO)$ on the ideal class group $\Cl(\cO)$. As explained in Section~\ref{par:polarizations}, the orbit $\mathscr P(\mathscr A)$ of the CM-action of $\mathscr P(\cO)$ on $\mathscr A$
is a set of $\F_q$-isomorphism classes of principally polarizable abelian surfaces isogenous to $\mathscr A$ and with same endomorphism ring $\cO$. This orbit contains \emph{all} such isomorphism classes when the CM-action is transitive, for instance when $\cO$ has maximal real multiplication (i.e., $\cO_{K_0} \subset \cO$).
Let $d_K$ be the discriminant of $K$. Applying Theorem~\ref{thm:rapidmixgeneral} allows to construct isogeny graphs on $\mathscr P(\mathscr A)$ that are expanders, and where all the isogenies are cyclic, with prime degrees bounded by $O\left(([\Cl(\cO) : \mathscr P(\cO)] \ln d_K)^{2 + \epsilon}\right)$. When $K$ is a primitive CM-field, the index $[\Cl(\cO) : \mathscr P(\cO)]$ is the narrow class number of $\cO \cap K_0$.

This result is used for two major applications, concerning the discrete logarithm problem in genus 2 and the computation of explicit isogenies between two isogenous principally polarized ordinary abelian surfaces. Aside from this, we remove certain heuristics from the complexity analysis of Galbraith's algorithm for elliptic curves. 

\subsubsection{Random self-reducibility of the discrete logarithm problem in genus 2.}
We use the rapid mixing properties of isogeny graphs to prove that the discrete logarithm problem in genus 2 is random self-reducible in isogeny subclasses of ordinary Jacobians of genus 2 curves over a finite field, thus extending the similar result for elliptic curves proved in \cite[Th. 1.6]{JMV09}.

\begin{thm}[Random self-reducibility in genus 2]\label{thm:randselfred}
Let $K$ be a primitive quartic CM-field, $K_0$ its maximal real subfield, and $\cO$ an order in $K$.
Let $\mathscr J$ be a Jacobian defined over $\F_q$ of endomorphism ring isomorphic to $\cO$.
Let $\mathcal V$ be the set of all $\F_q$-isomorphism classes of Jacobians defined over $\F_q$, isomorphic to $\mathscr J$ and with endomorphism ring isomorphic to $\cO$.
Let $G$ be a subgroup of $\mathscr J(\F_q)$ of order $Q$.
Suppose that
\begin{enumerate}
\item there is a polynomial time (in $\log q$) algorithm $\mathcal A$ that solves the DLP for a positive proportion $\mu > 0$ of the Jacobians in $\mathcal V$,
\item $\cO \cap K_0$ is the ring of integers of $K_0$, and $[\cO:\mathbf Z[\pi, \bar\pi]]$ is coprime to~$2Q$.
\end{enumerate}
Then, assuming the Generalized Riemann Hypothesis, there is an absolute polynomial $P$ in three variables such that the DLP can be solved on $G$ by a probabilistic algorithm of expected runtime 
$P(\log q, h_{\cO_{0}}^+, \mathrm{Disc}(K_0))/\mu$ , where $h_{\cO_{0}}^+$ is the narrow class number of the order $\cO_0 = \cO \cap K_0$. 
\end{thm}

\begin{rem}
In most practical applications, since the CM method is currently the only viable method to generate cryptographic parameters, both the narrow class number 
$h_{\cO_{0}}^+$ and the discriminant $\mathrm{Disc}(K_0)$ are small (constant or at most polynomial in $\log q$), and the above algorithm yields a polynomial (in $\log q$) reduction and 
thus, justifies the common cryptographic belief that the security of these curves is governed only by the characteristic polynomial of Frobenius. 
\end{rem}

\begin{rem}
The conditions that $\cO \cap K_0$ is the ring of integers of $K_0$, and $[\cO:\mathbf Z[\pi, \bar\pi]]$ is coprime to $2Q$ appear because they are required by the only currently known algorithm~\cite{dudeanu-jetchev-robert, Dudeanu16} to compute cyclic isogenies in genus 2.
\end{rem}

\subsubsection{Explicit isogenies in genus 2.}
In \cite{Gal99}, Galbraith considers the problem of computing an explicit isogeny between two isogenous ordinary elliptic curves 
$E_1$ and $E_2$ over $\F_q$. 
His approach is based on considering isogeny graphs and growing trees rooted at both $E_1$ and $E_2$ of small-degree computable isogenies until a collision is found. Galbraith's original algorithm is proven to finish in probabilistic polynomial time (in $\log q$), finding a path of length $O(\ln h_K)$ from $E_1$ to $E_2$, under GRH and a heuristic assumption claiming that the distribution of the new random points found in the process of growing the trees is close to uniform. 
In Section~\ref{sec:galbraith}, we use the expander properties of isogeny graphs to construct and analyze an algorithm similar to the one from \cite{Gal99}. This new algorithm improves upon Galbraith's in two ways. Firstly, its analysis relies only on GRH, without any additional heuristics. Secondly, it works in a generalized framework which, in particular, encompasses the case of elliptic curves, and of Jacobians of genus 2 hyperelliptic curves.

\subsection{Organization of the paper}\label{subsec:organization}
Section~\ref{sec:isographs} contains the necessary background on abelian varieties with complex multiplication, polarizations, and canonical lifting, and uses this theory to build the bridge between isogeny graphs and some Cayley graphs.
In Section~\ref{secExpanderCayley} we prove Theorem~\ref{thm:eigenCharacGeneral} and use it to prove Theorem~\ref{thm:rapidmixgeneral}.
In Section~\ref{sec:isoggraphsjac}, we discuss the consequences of these results on isogeny graphs of principally polarized abelian surfaces over finite fields and deduce Theorem~\ref{thm:randselfred}, the random self-reducibility.
Finally, we present the generalization of Galbraith's algorithm as well as the new complexity analysis in Section~\ref{sec:galbraith}.

\section{Isogeny graphs of ordinary abelian varieties}\label{sec:isographs}

In this section, we describe the relation between our graphs of interest -- graphs of horizontal isogenies between ordinary abelian varieties over finite fields -- and class groups of certain number fields, or subgroups thereof.

\subsection{Isogeny graphs over finite fields}

Let $\mathscr A$ be an absolutely simple, ordinary abelian variety of dimension $g$ over a finite field $\F_q$. Its endomorphism algebra $K = \End(\mathscr A) \otimes \Q$ is a CM-field, that is a totally imaginary quadratic extension of a totally real number field $K_0$. The field $K_0$ is of degree $g$ over $\Q$. The \emph{Frobenius polynomial} is the characteristic polynomial of the Frobenius endomorphism $\pi$ acting 
on the $\ell$-adic Tate module $T_\ell \mathscr A$ for $\ell$ different from the characteristic of $\F_q$. This endomorphism generates the field $K = \Q(\pi)$, and a theorem due to Tate~\cite{tate:isogeny} states that two abelian varieties defined over $\F_q$ are isogenous if and only if they have the same Frobenius polynomial. This element $\pi$ seen in $\overline{\Q}$ is a $q$-Weil number, and it uniquely determines the isogeny class of simple abelian varieties over $\F_q$ with Frobenius $\pi$~\cite[Lemma IV.2.2]{Streng10}.
The endomorphism ring of $\mathscr A$ is an order $\cO = \End(\mathscr A)$ in the CM-field $K$. We are interested in \emph{horizontal} isogeny graphs, i.e., graphs whose vertices are abelian varieties with the same endomorphism ring $\cO$ and whose edges are labelled by certain isogenies between these varieties.

The abelian varieties arising in cryptography are constructed as Jacobians of some hyperelliptic curves (usually of genus 1 or 2), and are therefore principally polarized. The case of elliptic curves is well understood and the literature on their isogeny graphs is already extensive. The present work aims at generalizing some of that literature, dealing with horizontal isogeny graphs, to other families of abelian varieties. We put a particular focus on principally polarized abelian surfaces, where these new results combined with the algorithm of~\cite{dudeanu-jetchev-robert,Dudeanu16} give rise to some interesting applications, yet the framework we develop is much more general.

\subsection{Class groups of orders}\label{subsec:ringclass} 
Class groups of orders in number fields are closely related to horizontal isogeny graphs, via the theory of complex multiplication, as will be recalled in Section~\ref{subsec:CMTheory}. In this subsection, we fix the notations and recall some useful results on class groups.

Let $K$ be a number field. Then, $\mathscr I(K)$ denotes the group of fractional ideals of $\cO_K$. Fix a modulus $\mathfrak m$, that is a formal product of primes in $K$, finite or infinite. The finite part is an ideal $\mathfrak m_0$ in $\cO_K$, and the infinite part is a subset $\mathfrak m_\infty$ of the real embeddings of $K$. Let $\mathscr I_\mathfrak m (K)$ be the subgroup generated by ideals coprime to $\mathfrak m_0$.
Let $P_{K,1}^\mathfrak m$  be the subgroup of $\mathscr I_\mathfrak m(K)$ generated by principal ideals of the form $\alpha \cO_K$ where $\ord_\mathfrak p(\alpha - 1) \geq \ord_\mathfrak p(\mathfrak m_0)$ for all primes $\mathfrak p$ dividing $\mathfrak m_0$, and $\imath(\alpha) > 0$ for all $\imath \in \mathfrak m_\infty$.
The \emph{ray class group} of $K$ modulo $\mathfrak m$ is the quotient group $\Cl_\mathfrak m(K) = \mathscr I_\mathfrak m(K)/ P_{K,1}^\mathfrak m$.
The \emph{narrow ray class group} modulo the ideal $\mathfrak m_0$ is $\Cl_\mathfrak m(K)$ when $\mathfrak m_\infty$ contains all the real embeddings.
\begin{example}
The subgroup $P_{K,1}^{\cO_K}$ is generated by all the principal ideals, so $\Cl_{\cO_K}(K)$ is the usual ideal class group $\Cl(K)$. Also, the narrow ray class group modulo $\cO_K$ is exactly the narrow class group $\Cl^+(K)$.
\end{example}

Let $\cO$ be an order in $K$. The \emph{conductor} of $\cO$, defined as ${\mathfrak f = \{x \in K \mid x\cO_K \subset \cO\}}$, is an invariant of the order. It is the largest subset of $K$ that is simultaneously an ideal in $\cO$ and in the maximal order $\cO_K$. An ideal in $\cO$ is invertible if and only if it is coprime to the conductor $\mathfrak f$.
Let $\mathscr I(\cO)$ denote the group of invertible ideals of $\cO$, and $P(\cO)$ the subgroup generated by principal ideals. The class group of $\cO$ is the quotient $\Cl(\cO) = \mathscr I(\cO) / P(\cO)$.
It can also be expressed as a quotient of $\mathscr I_\mathfrak f(K)$, as follows. Let $P_{K,\cO}^\mathfrak f$ be the subgroup of $\mathscr I_\mathfrak f(K)$ generated by principal ideals $\alpha \cO_K$ where $\alpha \in \cO$ and $\alpha\cO + \mathfrak f = \cO$. From~\cite[Th. 3.8]{Lv2015} and~\cite[Th. 3.11]{Lv2015}, the map sending any integral ideal $\mathfrak a$ of $\cO_K$ to the ideal $\mathfrak a \cap \cO$ of $\cO$ extends to a surjection $\mathscr I_\mathfrak f (K) \rightarrow \Cl(\cO)$ with kernel $P_{K,\cO}^\mathfrak f$.
Therefore, it induces an isomorphism $$ \Cl(\cO) \cong \mathscr I_\mathfrak f (K)/ P_{K,\cO}^\mathfrak f.$$
From~\cite[Th. 4.2]{Lv2015}, there is a unique abelian extension $H(\cO)$ of $K$, the \emph{ring class field} of $\cO$, such that all primes of $K$ ramified in $H(\cO)$ divide $\mathfrak f$, and the kernel of the Artin map $$\varphi_{H(\cO)/K}^{\mathfrak f}: \mathscr I_\mathfrak f (K) \rightarrow \Gal(H(\cO)/K)$$ is $P_{K,\cO}^\mathfrak f$. This map then induces an isomorphism $\Cl(\cO) \cong \Gal(H(\cO)/K)$. Similarly, there is a unique abelian extension $H^+(\cO)$, the \emph{narrow ring class field} of $\cO$, ramified only at primes dividing $\mathfrak f$ and at infinite primes, such that $\Gal(H^+(\cO)/K)$ is isomorphic to the narrow class group $\Cl^+(\cO)$, through the Artin map.

\subsection{Abelian varieties over $\C$ with CM}\label{subsec:CMTheory}
A key tool for studying isogeny graphs is the theory of complex multiplication (henceforth, CM theory) The main reference for this section is~\cite{taniyama-shimura}.
Let $\mathscr{A}_{\C} = \C^g / \Lambda$ be an abelian variety of dimension $g$ over $\C$, where $\Lambda$ is a lattice, that has complex multiplication by a CM-field $K$ and let $K_0$ be the real subfield of $K$ of degree $g$.

\subsubsection{CM-types.} The field $K$ has $2g$ embeddings in $\C$ which we denote $\vphi_1, \dots, \vphi_{2g}$. An endomorphism of $\mathscr{A}_{\C}$ yields an endomorphism of $\C^g$ and of $\Lambda$. We get an analytic representation $\rho_a \colon \End(\mathscr{A}_{\C}) \ra \End_{\C}(\C^g)$ and a rational  representation $\rho_{r} \colon \End(\mathscr{A}_{\C}) \ra \End_{\Z}(\Lambda)$. We have  
$\rho_r \otimes \C \sim \rho_a \oplus \overline{\rho}_a$ and at the same time, $\rho_r \otimes \C \sim \vphi_1 \oplus \dots \oplus \vphi_{2g}$. It follows that, up to some reindexing,  $\rho_a = \vphi_1 \oplus \dots \oplus \vphi_g$ where $\vphi_1,\dots, \vphi_g$ are not pairwise conjugate. We call $(K; \{\vphi_1,\dots, \vphi_g\})$ the CM-type of $\mathscr{A}_{\C}$.
The abelian variety $\mathscr A_\C$ is simple if and only if its CM-type is \emph{primitive}, which means that $(K; \Phi)$ is not a lift of a CM-type on a CM-subfield
of $K$~\cite[\S 8.2]{taniyama-shimura}. 

\begin{rem}
If $g = 2$, the abelian surface $\mathscr A_\C$ is simple if and only if the field $K$ is a primitive CM-field, i.e., $K$ does not have any proper CM-subfield. This follows from~\cite[Lemma I.3.4]{Streng10}.
\end{rem}

Fix 
a CM-type $\Phi = \{\varphi_1,\dots, \varphi_g\}$ for $K$. 
Any abelian variety over $\C$ of CM-type $(K; \Phi)$ is isomorphic to $\C^2 / \Phi(\mathfrak m)$ for some full-rank lattice $\mathfrak m$ in $K$, where $\Phi \colon K \ra \C^g$ is given by $x \mapsto (\vphi_1(x), \dots, \vphi_g(x))$. Let $\cO$ be the order of $K$ isomorphic to the endomorphism ring of the variety. Then, the lattice $\mathfrak m$ is an $\cO$-submodule of $K$, and $\cO$ coincides with the order $\cO(\mathfrak m)$ associated to the lattice,
$$\cO(\mathfrak m) = \{\alpha \in K \mid \alpha \mathfrak m \subset \mathfrak m\}.$$
Given an ideal $\mathfrak a$ in $\cO$, the variety $\C^2 / \Phi(\mathfrak a^{-1}\mathfrak m)$ is isogenous to $\C^2 / \Phi(\mathfrak m)$, and its endomorphism ring is also $\cO$. This isogenous variety is actually isomorphic if and only if $\mathfrak a$ is principal. In fact, this construction induces a free action of the ideal class group $\Cl(\cO)$ on the set of isomorphism classes of abelian varieties of CM-type $(K; \Phi)$ with endomorphism ring $\cO$.

\subsubsection{Polarizations and the Shimura class group.}\label{par:polarizations} A polarization on an abelian variety $X$ over a field $k$ is an ample line bundle $\cL_X$ on $X$. 
Associated to such $\cL_X$ is the polarization isogeny $\vphi_{\cL_X} \colon X \ra X^\vee$, where $ X^\vee$ is the dual of $X$. A principal polarization is an ample 
line bundle of degree one (equivalently, the polarization isogeny is an isomorphism).
\begin{example}
If $\mathscr A_\C$ is a simple abelian surface, $\mathscr A_\C$ is principally polarizable if and only if it is the Jacobian of a genus 2 curve (see~\cite[Prop. 3.13]{milne:cm} and~\cite[Th.~4.1]{MN02}).
\end{example}
In the remainder of this paragraph, we shall restrict to simple abelian varieties, or equivalently, to primitive CM-types $(K; \Phi)$.
If $X = \mathscr{A}_{\C}$, a simple complex 
abelian variety with CM by an order $\cO$ in $K$, the theory of Taniyama and Shimura \cite[\S 14]{taniyama-shimura} which we now briefly recall provides 
an explicit description of the polarizations on $X$ in terms of the arithmetic of $K$. Indeed, by the theory of 
complex multiplication, there exists a full-rank lattice $\mathfrak m$ in $K$ such that 
$X(\C) \isom \C^2 / \Phi(\mathfrak m)$. The dual abelian variety of $\C^g / \Phi(\mathfrak m)$ is $\C^g / \Phi(\mathfrak m^*)$ where
$\mathfrak m^* = \{ \beta \in K \colon \Tr_{K/\Q}(\beta \overline{\mathfrak m})
\in \Z \}$.  
A polarization $\cL$ on $\C^g / \Phi(\mathfrak m)$ induces
an isogeny $\vphi_{\cL} \colon \C^g / \Phi(\mathfrak m) \ra
\C^g / \Phi(\mathfrak m^*)$ that is given by $x \mapsto \rho_a(\xi) x$
for some purely imaginary element $\xi \in K$ that satisfies $\Phi(\xi) \in (i\R_{>0})^2$.
The polarization is also described by the Riemann form
$E(x,y)=\Tr_{K/\Q}(\xi \overline{x} y)$.
The polarization is principal if and only if
$\vphi_{\cL}(\Phi(\mathfrak m)) = \Phi(\mathfrak m^*)$, i.e., if and only if 
$\xi \mathfrak m = \mathfrak m^*$.
Thus, the CM-type $(K; \Phi)$ being fixed, the principally polarized abelian variety $(\mathscr A_\C, \mathcal L)$ is determined by the pair $(\mathfrak m, \xi)$.
The \emph{Shimura class group} of $\cO$, acts on such pairs. 
It is defined as
$$
\mathfrak C(\cO) = \{(\mathfrak a,\alpha) \mid \mathfrak a \in \mathscr I(\cO) \text{ and } \mathfrak a\overline{\mathfrak a} = \alpha\cO, \alpha \in K_0\text{ totally positive}\}/\sim
$$
with componentwise multiplication, where two pairs $(\mathfrak a,\alpha)$ and $(\mathfrak b,\beta)$ are equivalent for the relation~$\sim$ if there exists an element $u \in K^\times$ such that $\mathfrak b = u\mathfrak a$ and $\beta = u\overline u\alpha$.
For any $(\mathfrak a,\alpha) \in \mathfrak C(\cO)$ (up to equivalence), the pair $(\mathfrak a^{-1}\mathfrak m, \alpha\xi)$ corresponds to a principally polarized abelian variety isogenous to $\mathscr A_\C$ and with same endomorphism ring $\cO$ (up to isomorphism).
This action of $\mathfrak C(\cO)$ is in fact free on the set of isomorphism classes of principally polarized abelian varieties isogenous to $\mathscr A_\C$ with same endomorphism ring~\cite[\S 17]{taniyama-shimura}.
The structure of $\mathfrak C(\cO)$ and its relation to $\Cl(\cO)$ is described by the exact sequence
$$1 \longrightarrow (\cO_0^\times)^+/N_{K/K_0}(\cO^\times) \xrightarrow{u\mapsto (\cO,u)} \mathfrak C(\cO) \xrightarrow{(\mathfrak a,\alpha)\mapsto \mathfrak a} \Cl(\cO) \xrightarrow{N_{K/K_0}} \Cl^+(\cO_0),$$
where $\cO_0 = \cO \cap K_0$, $(\cO_0^\times)^+$ is its subgroup of totally positive units, and $\Cl^+(\cO_0)$ its narrow class group.
\comment{Since the CM-field $K$ is primitive, the structure of $\mathfrak C(K)$ and its relation to $\Cl_K$ is described by the exact sequence \cite[Theorem 3.1]{gruenewald-broeker-lauter}
$$1 \longrightarrow (\cO_{K_0}^\times)^+/N_{K/K_0}(\cO_{K}^\times) \xrightarrow{u\mapsto (\cO_K,u)} \mathfrak C(K) \xrightarrow{(\mathfrak a,\alpha)\mapsto \mathfrak a} \Cl_K \xrightarrow{N_{K/K_0}} \Cl_{K_0}^+ \longrightarrow 1.$$
\bennote{Check that it indeed generalizes to arbitrary degree primitive CM-fields.}
}The image of the projection $\mathfrak C(\cO) \rightarrow \Cl(\cO)$, denoted $\mathscr P(\cO)$, is a subgroup of $\Cl(\cO)$ that acts freely on the set of principally \emph{polarizable} abelian varieties isogenous to $\mathscr A_\C$ with endomorphism ring $\cO$. 
Notice the crucial distinction between polarized and polarizable. The amount of information lost with the polarization is encoded in the group $(\cO_0^\times)^+/N_{K/K_0}(\cO^\times)$. For a maximal orders in quartic CM-fields, this group is either trivial, in which case $\mathfrak C(\cO)$ and $\mathscr P(\cO)$ are isomorphic and no information is lost, or it is of order two, in which case the abelian surfaces encoded in $\mathscr P(\cO)$ each have two possible polarizations. From the exactness of the sequence, the subgroup $\mathscr P(\cO)$ is also the kernel of $N_{K/K_0}$.
The following lemma allows to extend the result of~\cite[Th. 3.1]{gruenewald-broeker-lauter} to higher dimensions, and non-maximal orders.
\begin{lem}
Let $K$ be a CM-field and $K_0$ its maximal real subfield. Let $\cO$ be an order in $K$ of conductor $\mathfrak f$, and $\cO_0 = \cO \cap K_0$.
The index of the image of $\Cl(\cO)$ through the norm map $N_{K/K_0} : \Cl(\cO) \rightarrow \Cl^+(\cO_0)$ is of index at most 2 in $\Cl^+(\cO_0)$.
If there is a prime in $K_0$ that ramifies in $K$ and does not divide $\mathfrak f$, the norm map $N_{K/K_0}$ is surjective.
\end{lem}

\begin{proof}
We use the elements of class field theory recalled in Section~\ref{subsec:ringclass}.
Let $H = H(\cO)$ and $H^+ = H^+(\cO_0)$. The compositum $KH^+$ is a subfield of $H$, so we have a natural surjection $\Gal(H/K) \rightarrow \Gal(KH^+/K)$. From Galois theory, $\Gal(KH^+/K)$ is isomorphic to $\Gal(H^+/(K\cap H^+))$, which in turn is isomorphic to the quotient $\Gal(H^+/K_0)/\Gal((K\cap H^+)/K_0)$. Let $N = \Gal((K\cap H^+)/K_0)$. Then,
$$\psi : \Gal(H/K) \longrightarrow \Gal(H^+/K_0)/N : \sigma \longmapsto \sigma\mid_{H^+} \mod N,$$
is the composition of these canonical maps, and is therefore is a surjection. Through the Artin map, the norm $N_{K/K_0}$ commutes with $\psi$. We conclude that the image of $\Cl(\cO)$ through $N_{K/K_0}$ is a subgroup of $\Cl^+(\cO_0)$ of index at most $|N| \leq 2$. If there is a prime in $K_0$ that ramifies in $K$ and does not divide $\mathfrak f$, then $K\cap H^+ = K_0$, so $|N| = 1$ and the map $N_{K/K_0}$ is surjective.
\end{proof}

In particular, this lemma implies that the index $[\Cl(\cO) : \mathscr P(\cO)]$ is either the narrow class number $h_{\cO_0}^+ = |\Cl^+(\cO_0)|$, or $h_{\cO_0}^+/2$. It is exactly $h_{\cO_0}^+$ whenever there is a prime in the field $K_0$ that ramifies in $K$ and does not divide $\mathfrak f$. As observed in~\cite[Th. 3.1]{gruenewald-broeker-lauter}, there exists such a prime when $\cO$ is the maximal order in a primitive quartic CM-field.

\subsection{Canonical lifting}\label{subsec:canlift}
Recall that our objects of primary interest are varieties defined over a finite field $\F_q$.
The theory of canonical lifting of Serre and Tate \cite{serre-tate} allows us to lift an ordinary abelian variety $\mathscr A/\F_q$ to an abelian variety $\tilde{\mathscr A}$ over $W(\mathbf F_q)$, the ring of Witt vectors of $\F_q$ in such a way that all endomorphisms of $\mathscr A$ lift to endomorphisms of $\tilde{\mathscr A}$, and $\mathscr A \mapsto \tilde{\mathscr A}$ is functorial. To obtain lifts from abelian varieties over $\F_q$ to abelian varieties over 
$\C$, we fix an embedding $\imath \colon W(\Fbar_q) \hookrightarrow \mathbf C$ and let $\mathscr{A}_{\mathbf C}$ be the complex abelian variety $\tilde{\mathscr A}\otimes_{\imath}\mathbf C$. If $T(\mathscr A) = H_1(\mathscr{A}_{\C}, \Z)$ then $T(\mathscr A)$ is a free $\Z$-module of rank $2\cdot \dim(\mathscr A)$. The correspondence $\mathscr A \mapsto T(\mathscr A)$ is functorial and any isogeny $\varphi \colon \mathscr A \rightarrow \mathscr B$ over $\Fbar_q$ gives rise to a short exact sequence
$$
0 \longrightarrow T(\mathscr A) \xrightarrow{T(\varphi)} T(\mathscr B) \longrightarrow \ker(\varphi) \longrightarrow 0.
$$
A theorem of Deligne \cite[Th.7]{Del69} says that if $\pi$ is the Frobenius endomorphism of $\mathscr A$ over $\mathbf F_q$ then the functor $\mathscr A \mapsto (T(\mathscr A), T(\pi))$ is an equivalence of categories between the category of ordinary abelian varieties over $\mathbf F_q$ and the category of free $\mathbf Z$-modules $T$ endowed with an endomorphism $F$ satisfying
\begin{enumerate}
\item $F$ is semi-simple, with eigenvalues of complex absolute value $\sqrt{q}$,
\item At least half the roots in $\overline{\mathbf Q}_p$ of the characteristic polynomial of $F$ are $p$-adic units,
\item There is an endomorphism $V$ of $T$ such that $FV = q$.
\end{enumerate}
As discussed in \cite[\S 8]{Del69}, any such $(T,F)$ that is the image of a variety $\mathscr A$ through this functor determines the complex abelian variety $\mathscr A_{\C}$ up to isomorphism as $\mathscr A_{\C} \cong (T\otimes\mathbf R)/T$ (with a complex structure on $T\otimes\mathbf R$ such that $F$ is $\C$-linear; the existence and uniqueness of the appropriate complex structure is established by a theorem of Serre \cite[\S 8]{Del69}). This means that up to isomorphism, we can write $\mathscr{A}_{\mathbf C} = \mathbf C^g / \Lambda$, for a lattice $\Lambda$ in $\mathbf C^2$ and since  lifting preserves the endomorphism ring $\cO = \End(\mathscr A)$, we even have $\mathscr{A}_{\mathbf C} = \mathbf C^g / \Phi(\mathfrak m)$ for some full-rank lattice $\mathfrak m$ in $K$ with order $\cO(\mathfrak m) = \cO$, where, as above, the map $\Phi \colon K \ra \C^g$ is the CM-type of $\mathscr{A}_{\mathbf C}$. From the canonical identification between $\Phi(\mathfrak m)$ and $\H_1(\mathscr{A}_{\mathbf C}, \mathbf Z)$ (see \cite[\S 1.1]{BL04}), the functor can be interpreted as $\mathscr A \mapsto (\Phi(\mathfrak m), \rho_r(\pi))$.
This establishes a functorial map from the abelian varieties over $\mathbf F_q$ of fixed endomorphism ring $\cO$ to the complex abelian varieties $\mathbf C^g / \Phi(\mathfrak m)$ where $\mathfrak m$ are lattices in $K$ with order~${\cO(\mathfrak m) = \cO}$. Conversely, Deligne's theorem shows that any such $\mathbf C^g / \Phi(\mathfrak m)$ is the lift of an abelian variety over $\mathbf F_q$ with endomorphism ring $\cO$: the variety corresponding to the pair $(\Phi(\mathfrak m), \rho_r(\pi))$, where $\rho_r(\pi)$ is the rational representation of $\pi$. Moreover, from \cite[\S 3]{Del69}, the polarizations also lift properly, and in particular $\mathscr A$ is principally polarizable if and only if $\mathscr{A}_{\mathbf C}$ is.

\subsection{Horizontal isogeny graphs as Cayley graphs}\label{sec:isoGraphsCayleyGraphs}
Let $\pi$ be a $q$-Weil number, and let $K = \Q(\pi)$ be the corresponding CM-field,with $K_0$ its maximal real subfield. Fix an order $\cO$ in $K$, and let $\mathcal V_{\pi,\cO}$ be the set of all $\F_q$-isomorphism classes of abelian varieties defined over $\mathbf F_q$ with endomorphism ring $\cO$ in the isogeny class characterised by $\pi$. Recall that the class group $\Cl(\cO)$ acts freely on $\mathcal V_{\pi,\cO}$. One can choose any reference variety $\mathscr A$ in $\mathcal V_{\pi,\cO}$ and any subgroup $H$ in $\Cl(\cO)$, and consider the orbit $H(\mathscr A)$.

Combining the results of Deligne discussed in Section~\ref{subsec:canlift} with the theory of complex multiplication, there is an equivalence of categories between the category of objects $H(\mathscr A)$ and morphisms the isogenies between them, and the category whose objects are the ideal classes in the subgroup $H$, and the sets of morphisms from $a \in H$ to $b \in H$ are the ideals of $\cO$ in the class $a^{-1} b$. The degree of an isogeny equals the norm of the corresponding ideal. Restricting the morphisms to a finite set of generators, the latter category can be seen as a Cayley (multi)graph.

\begin{defn}[Cayley graph]\index{Cayley graph}
Let $G$ be a finite group and $S$ a generating subset of $G$, with $S = S^{-1}$. The \emph{Cayley graph} $\Cay(G,S)$ is the finite $|S|$-regular undirected graph with set of vertices $G$, and an edge between $g$ and $sg$ for any $g\in G$ and $s \in S$.
\end{defn}

\begin{rem}
The edges of $\Cay(G,S)$ can have multiplicities if $S$ is a multiset. If $\mathcal S$ is a set of labels and $f : \mathcal S \rightarrow S$ is a surjection, then $f$ naturally induces a Cayley multigraph for the set of generators $S$ whose edges are labelled by elements of $\mathcal S$.

\end{rem}

Let $\mathcal S$ be a set of ideals of $\cO$, and $S$ its image in $\Cl_K$, with $f : \mathcal S \rightarrow S$ the induced surjection. Let $\Cay(H, S \cap H)$ be the induced labelled multigraph. Let $T$ be the set of all isogenies between elements of $H(\mathscr A)$ corresponding to the ideals of $\mathcal S$. We build the graph $\mathscr G_{\mathcal S}$ with set of vertices $H(\mathscr A)$ by adding an edge between the vertices $\mathscr B$ and $\mathscr C$ for any isogeny $\mathscr B \rightarrow \mathscr C$ in $T$. Then, the equivalence of categories induces an isomorphism between the graphs $\mathscr G_{\mathcal S}$ and $\Cay(H, S \cap H)$.

\begin{example}\label{example:ppav}
If $\mathscr A$ is a principally polarizable abelian variety and ${H = \mathscr P(\cO)}$, 
the orbit $H(\mathscr A)$ (in this case also denoted $\mathscr P(\mathscr A)$) is a set of isomorphism classes of principally polarizable abelian varieties isogenous to $\mathscr A$ and with same endomorphism ring.
Via the construction described above, any choice of a generating set of $\mathscr P(\cO)$ yields a graph of
the set of vertices $\mathscr P(\mathscr A)$.
From~\cite[Theorem 5.3]{Waterhouse1969} together with~\cite[Theorem 4.5]{blake2014deuring}, the action of $\Cl(\cO)$ is transitive on the set of \emph{all} abelian varieties isogenous to $\mathscr A$ and with same endomorphism ring whenever $\mathscr A$ has maximal real multiplication (i.e., $\cO_{K_0} \subset \cO$). 
We can conclude via~\cite[\S 17]{taniyama-shimura} that when $\mathscr A$ has maximal real multiplication, the orbit $\mathscr P(\mathscr A)$ is exactly the set of all isomorphism classes of principally polarizable abelian varieties isogenous to $\mathscr A$ and with same endomorphism ring.
\end{example}

\section{Expander graphs and ray class groups}\label{secExpanderCayley}

In this section, we prove Theorem~\ref{thm:eigenCharacGeneral}, and investigate its consequences on the structure of the Cayley graphs of interest.

\subsection{Eigenvalues and Cayley graphs}\label{subsecCayleyRay}

Let $\mathscr G$ be an undirected (multi)graph with set of vertices $\mathcal V$ and set of edges $\mathcal E$. Suppose $\mathscr G$ is finite and $k$-regular, i.e., each vertex has $k$ incident edges. The \emph{adjacency operator}\index{adjacency operator} $A$ of $\mathscr G$ is the operator defined for any function $f$ from $\mathcal V$ to $\C$ by
$$Af(x) = \sum_{y \in \mathcal N_{\mathscr G}(x)}f(y),$$
for any $x \in \mathcal V$, where $\mathcal N_{\mathscr G}(x)$ denotes the (multi)set of neighbors of $x$ in $\mathscr G$.
This operator is represented by the adjacency matrix of $\mathscr G$ with respect to the basis $\{\mathbf 1_{\{x\}} : x \in \mathcal V\}$, where $\mathbf 1_S$ denotes the characteristic function of a set $S$. It is a real symmetric matrix, so by the spectral theorem, $A$ has $n = |\mathcal V|$ real eigenvalues $\lambda_1 \geq \lambda_2 \geq ... \geq \lambda_n$. Since the graph is $k$-regular, the constant function $\mathbf 1_{\mathcal V} : x \mapsto 1$ is an eigenvector with eigenvalue $k$. We call $k$ the \emph{trivial eigenvalue}\index{trivial eigenvalue}, and denote it by $\lambda_{\mathrm{triv}}$. This $\lambda_{\mathrm{triv}}$ is the largest eigenvalue in absolute value, i.e., $\lambda_1 = k$, and its multiplicity is the number of connected components of $\mathscr G$.

\begin{defn}[Expander graph]\index{expander graph}
Let $\delta > 0$. The $k$-regular graph $\mathscr G$ is \emph{(one-sided) $\delta$-expander} if $\lambda_2 \leq (1-\delta)\lambda_{\mathrm{triv}}$. It is a \emph{two-sided $\delta$-expander} if the stronger bound $|\lambda_2| \leq (1-\delta)\lambda_{\mathrm{triv}}$ holds.
\end{defn}

Observe that such a graph is connected whenever $\delta > 0$. The main reason for our interest in expander graphs is that they rapidly mix random walks. The following lemma is a classical result on expander graphs and can be found in, e.g., \cite{JMV09}.

\begin{lem}\label{lemmaRapidMix}
Let $\mathscr G$ be a finite $k$-regular graph for which the non-trivial eigenvalues $\lambda$ of the adjacency operator $A$ satisfy the bound $|\lambda| \leq c$, for some $c < k$. Let $S$ be a subset of the vertices of $\mathscr G$, and $v$ a vertex of $\mathscr G$. Any random walk from $v$ of length at least $\ds \frac{\ln(2|\mathscr G|/|S|^{1/2})}{\ln(k/c)}$ will end in $S$ with probability between $\ds \frac{1}{2}\frac{|S|}{|\mathscr G|}$ and $\ds \frac{3}{2}\frac{|S|}{|\mathscr G|}$.
\end{lem}

For any finite group $G$ with generating set $S$, observe that a character ${\chi : G \rightarrow \C^*}$ is an eigenvector for the adjacency operator $A$ on $\Cay(G,S)$. Indeed,
$$A\chi(x) = \sum_{s\in S}\chi(sx) = \sum_{s\in S}\chi(s)\chi(x) = \lambda_{\chi}\chi(x), \text{ where }\lambda_\chi = \sum_{s\in S}\chi(s).$$

If $G$ is abelian, these characters form a basis of the $\C$-vector space of functions of $G$. In particular, any eigenvalue is of the form $\lambda_\chi$ for some character $\chi$. The trivial eigenvalue corresponds to the trivial character $\mathbf 1_G$.

\subsection{Proof of Theorem \ref{thm:eigenCharacGeneral}}
Since $G$ is abelian, any character $\chi$ of $H$ can be extended to a character of $G$. Take any  such extension and, by abuse of notation, also denote it by $\chi$. Note that for any ideal $\frakl$ of $\cO_K$ coprime to $\mathfrak m$, we have
$$
\sum_{\theta \in \widehat{G/H}} \theta([\frakl]H) = 
\begin{cases}
[G:H] & \text{if } [\frakl] \in H,\\ 
0 & \text{otherwise},  
\end{cases}
$$
where $\widehat{G/H} = \Hom(G/H, \C^*)$ is the character group of the quotient $G/H$. 
Therefore this sum can be used to filter the condition that $[\frakl] \in H$, and we can rewrite 
\begin{alignat*}{1}
\lambda_{\chi} &= \sum_{\frakl \in \cT_{H,\mathfrak m}(B)} \chi([\frakl])
 = \frac{1}{[G:H]} \sum_{\substack{\frakl: N\frakl < B \text{ prime}\\ (\frakl, \mathfrak m) = 1}} \chi([\frakl]) \sum_{\theta \in \widehat{G/H}} \theta([\frakl]H)\\
& = \frac{1}{[G:H]}  
\sum_{\theta \in \widehat{G/H}} \sum_{\substack{\frakl : N\frakl < B \text{ prime} \\ (\frakl, \mathfrak m) = 1}} \chi([\frakl])\theta([\frakl ]H). 
\end{alignat*}
We are then left with estimating a character sum $\sum\chi([\frakl])\theta([\frakl]H)$. Each of the summands of the latter defines a multiplicative function 
\begin{alignat*}{1}
\nu_{\chi, \theta} \colon \mathscr I_{\mathfrak m}(K) & \longrightarrow \C^* : \frakl  \longmapsto \chi([\frakl])\theta([\frakl]H)
\end{alignat*}
where 
$\mathscr I_{\mathfrak m}(K)$ is the group of fractional ideals of $K$ coprime to $\mathfrak m$. It extends to a function of $\mathscr I(K)$, the group of all the fractional ideals of $K$, by setting $\nu_{\chi, \theta}(\frakl) = 0$ for all prime divisors $\frakl$ of $\mathfrak m$.
The expression of ${\lambda}_{\chi}$ becomes
\begin{equation}\label{eq:formulaLambdaChi}
{\lambda}_{\chi} = \frac{1}{[G:H]} 
\sum_{\theta \in \widehat{G/H}} \sum_{\frakl : N\frakl < B \text{ prime}} \nu_{\chi, \theta}(\frakl).
\end{equation}

From the classical estimate that can be found in \cite[Th.5.15]{IK04}, we have
$$\sum_{\mathfrak a : N\mathfrak a < B} \Lambda(\mathfrak a)\nu_{\chi, \theta}(\mathfrak a) = \delta(\nu_{\chi, \theta})B + O\left(nB^{1/2}\ln(B)\ln(Bd_KN\mathfrak m)\right),$$
where $\Lambda$ is the von Mangoldt function (i.e., $\Lambda(\mathfrak a)$ is $\ln N\mathfrak l$ if $\mathfrak a$ is a power of a prime ideal $\mathfrak l$, and 0 otherwise), and $\delta(\nu_{\chi, \theta})$ is 1 if $\nu_{\chi, \theta}$ is principal, and 0 otherwise (a principal character is a character that only takes the values 1 or 0). 
Observe that if $\nu_{\chi, \theta}$ is principal, then $\chi$ must be the trivial character, so that $\delta(\nu_{\chi, \theta}) = \delta(\chi) \delta(\theta)$.
Indeed, suppose that $\nu_{\chi, \theta}$ is principal, and let $[\frakl] \in H$, for a prime $\frakl$ coprime to $\mathfrak m$. Then,
$$1 = \nu_{\chi, \theta}(\frakl) = \chi([\frakl])\theta([\frakl]H) = \chi([\frakl])\theta(1_{G/H}) = \chi([\frakl]),$$
so $\chi$ must be the trivial character of $H$.

We now want to replace each instance of $\Lambda(\mathfrak a)$ in the above sum by $P(\mathfrak a)$, where 
$$
P(\mathfrak a) = 
\begin{cases}
\ln N\mathfrak a & \text{if $N\mathfrak a$ is prime}, \\ 
0 & \text{otherwise}.
\end{cases}
$$
To do so, it is sufficient to prove that
\begin{equation}\label{eq:diffMangPrime}
\sum_{\mathfrak a :N\mathfrak a <B}\Lambda(\mathfrak a)\nu_{\chi, \theta}(\mathfrak a) - \sum_{\mathfrak a :N\mathfrak a <B}P(\mathfrak a)\nu_{\chi, \theta}(\mathfrak a) = O\left(nB^{1/2}\right).
\end{equation}
The non-zero terms $(\Lambda(\mathfrak a) - P(\mathfrak a))\nu_{\chi, \theta}(\mathfrak a)$ correspond to ideals $\mathfrak a$ which are powers of a prime ideal $\mathfrak l$, and $N\mathfrak a = N\mathfrak l^k$ is not a prime number -- but it is a power of a prime $\ell$. Since $K$ is of degree $n$, there are at most $n$ different prime ideals $\mathfrak l$ above any given prime number $\ell$. Therefore the difference~\eqref{eq:diffMangPrime} is bounded in absolute value by
$$n\sum_{\substack{\ell^k < B\\ k \geq 2}} \ln \ell = n\sum_{\substack{\ell < B^{1/2}\\2 \leq k < \frac{\ln B}{\ln \ell}}} \ln \ell \leq n\sum_{\ell < B^{1/2}} \ln \ell\frac{\ln B}{\ln \ell} = n\pi(B^{1/2})\ln B,$$
which, by the Prime Number Theorem, is $O(nB^{1/2})$. Therefore,
$$\sum_{\mathfrak a : N\mathfrak a < B} P(\mathfrak a)\nu_{\chi, \theta}(\mathfrak a) = \delta(\nu_{\chi, \theta})B + O\left(nB^{1/2}\ln(B)\ln(Bd_KN\mathfrak m)\right).$$
Applying the Abel partial summation formula, we derive that
$$\sum_{\frakl : N\frakl < B\text{ prime}} \nu_{\chi, \theta}(\frakl) = \delta(\nu_{\chi, \theta})\mathrm{li}(B) + O\left(nB^{1/2}\ln(Bd_KN\mathfrak m)\right),$$
where $\mathrm{li}$ denotes the logarithmic integral. Replacing this into the expression~\eqref{eq:formulaLambdaChi} of ${\lambda}_{\chi}$, we finally obtain
$${\lambda}_{\chi} = \frac{\delta(\chi)}{[G:H]}\mathrm{li}(B) + O\left(nB^{1/2}\ln(Bd_KN\mathfrak m)\right),$$
which proves the theorem.\qed

\subsection{Spectral gaps for subgroups of ideal class groups}\label{subsec:spectral-quartic}
Let $K$ be any number field of degree $n$, $\cO$ an order of conductor $\mathfrak f$ in $K$, and $H$ any subgroup of $\Cl(\cO)$. Let $B > 0$, $\mathfrak m$ an integral ideal of $\cO_K$, and define the following set of ideals of $\cO_K$, 
$$
\cS_B= \{\frakl \mid N\frakl < B\text{ is prime}, (\frakl, \mathfrak f \mathfrak m) = 1, \text{ and }[\frakl\cap \cO] \in H\}, 
$$
where $[\frakl\cap \cO]$ is the class in $\Cl(\cO)$. Let $S_B$ be the multiset of its image in the class group. Using Theorem \ref{thm:eigenCharacGeneral}, one can bound the spectral gap of $\mathscr G_B = \Cay( H, S_B)$.

\begin{thm}\label{thm:eigenCharac}
For any character $\chi$ of $H$,the corresponding eigenvalue of $\mathscr G_B$ is
$${\lambda}_{\chi} = \frac{\delta(\chi)}{[\Cl(\cO):H]}\mathrm{li}(B) + O(nB^{1/2}\ln(Bd_KN(\mathfrak f\mathfrak m))),$$
where $\delta(\chi)$ if $1$ if $\chi$ is trivial, and $0$ otherwise.
\end{thm}
\begin{proof}
Using the notations from Section~\ref{subsec:ringclass}, the group $P_{K,1}^\mathfrak f$ is a subgroup of $P_{K,\cO}^\mathfrak f$, so there is a natural surjection $\Cl_\mathfrak f(K) \rightarrow \Cl(\cO)$. Furthermore, the canonical injection of $\mathscr I_{\mathfrak f\mathfrak m} (K)$ in $\mathscr I_\mathfrak f (K)$ induces a surjection from $\Cl_{\mathfrak f\mathfrak m}(K)$ to $\Cl_\mathfrak f(K)$. Therefore we have a natural surjection $\pi : \Cl_{\mathfrak f\mathfrak m}(K) \rightarrow \Cl(\cO)$, which sends the class of any integral ideal $\mathfrak a$ of $\cO_K$ to the class of $\mathfrak a \cap \cO$.
Consider the subgroup $\widetilde{H} = \pi^{-1}(H)$ of $\Cl_{\mathfrak f \mathfrak m}(K)$, and its Cayley graph $\widetilde{\mathscr G}_B = \Cay(\widetilde{H}, T_{\widetilde{H}, \mathfrak f\mathfrak m}(B))$ where $T_{\widetilde{H}, \mathfrak f \mathfrak m}(B)$
 is the multiset defined in the statement of Theorem~\ref{thm:eigenCharacGeneral}.
The Cayley graph $\mathscr G_B$ on $H$ is the image of the Cayley graph $\widetilde{\mathscr G}_B$ on $\widetilde H$ via the projection
 $\pi$, taking into account the multiplicity of the edges. The eigenvalues of 
 $\mathscr G_B$ are exactly the eigenvalues $\lambda_{\theta}$ of $\widetilde{\mathscr G}_B$ corresponding to characters $\theta$ of $\widetilde H$ that are trivial on the 
 kernel of $\pi|_{\widetilde H} : \widetilde H \rightarrow H$. The result follows by applying Theorem \ref{thm:eigenCharacGeneral} on $\widetilde{\mathscr G}_B$.
\end{proof}
\begin{cor}\label{coroBoundExpander}
For any $0 < \delta < 1$ and $\varepsilon > 0$, there is a function 
$$B_{\delta,\varepsilon}(H, \mathfrak m) = O\left( \left(n [\Cl(\cO):H] \ln(d_K N (\mathfrak f \mathfrak m))\right)^{2 + \varepsilon}\right),$$
such that $\mathscr G_{B_{\delta,\varepsilon}(H, \mathfrak m)}$ is a two-sided $\delta$-expander.
\end{cor}

\begin{proof}
Let $x > 0$, and write $k = [\Cl(\cO):H]$. The graph $\mathscr G_x$ is a two-sided $\delta$-expander if
$|\lambda_\chi| \leq (1-\delta)\lambda_{\mathrm{triv}}$ for any non-trivial character $\chi$. From Theorem \ref{thm:eigenCharac}, and the fact that $\mathrm{li}(x) \sim x/\ln(x)$ and $\mathrm{li}(x) \geq x/\ln(x)$ for any $x \geq 4$, there are absolute constants $C$ and $D$ such that for any $x \geq C$, we have $$\lambda_{\mathrm{triv}} \geq \frac{x}{\ln(x)k} - Dnx^{1/2}\ln(xd_KN(\mathfrak f\mathfrak m)),$$ and $|\lambda| \leq Dnx^{1/2}\ln(xd_KN(\mathfrak f\mathfrak m))$. So
\begin{alignat*}{1}
\frac{\lambda_{\mathrm{triv}}}{|\lambda|} & \geq \frac{2x^{1/2}}{\ln^2(x)D k n (\ln(d_KN(\mathfrak f\mathfrak m))+1)} - 1.
\end{alignat*}
We have that $x^{1/(2+\varepsilon)} = O(x^{1/2}/\ln^2(x))$ for any $\varepsilon > 0$, so considering larger constants $C$ and $D$ if necessary, we have the inequality
$$\frac{\lambda_{\mathrm{triv}}}{|\lambda|} \geq \frac{2x^{1/(2+\varepsilon)}}{D k n (\ln(d_KN(\mathfrak f\mathfrak m))+1)} - 1.$$
The constants $C$ and $D$ are not absolute anymore but they only depend on $\varepsilon$. Let
$$B_{\delta,\varepsilon}(H, \mathfrak m) = \max\left(C, \left(\frac{1}{2}\left(\frac{1}{1-\delta} + 1\right)(D k n (\ln(d_KN(\mathfrak f\mathfrak m))+1)\right)^{2 + \varepsilon}\right).$$
Then, for $x = B_{\delta,\varepsilon}(H, \mathfrak m)$, we have
$\frac{\lambda_{\mathrm{triv}}}{|\lambda|} \geq \frac{1}{1-\delta},$
so $\mathscr G_x$ is $\delta$-expander.
\end{proof}

\subsection{Proof of Theorem~\ref{thm:rapidmixgeneral}} Theorem~\ref{thm:rapidmixgeneral} is now an easy combination of the graph isomorphism expounded in Section~\ref{sec:isoGraphsCayleyGraphs}, together with Corollary~\ref{coroBoundExpander} establishing that these graphs are expanders, and Lemma~\ref{lemmaRapidMix} on random walks on such graphs.

\section{Random walks on isogeny graphs of Jacobians in genus 2}\label{sec:isoggraphsjac}

Throughout this section, we will restrict to ordinary abelian surfaces that are Jacobians of genus 2 hyperelliptic curves over a finite field $\F_q$.
Let $\mathscr J = \Jac(\mathscr C)$ be such a Jacobian with endomorphism algebra $K$ and whose endomorphism ring is isomorphic to an order $\cO$ in $K$. Let $\cO_0 = \cO \cap K_0$ where $K_0$ is the real subfield of $K$. Let $\mathscr A$ be the isomorphism class of $\mathscr J$ as an abelian variety.

Consider the orbit $\mathscr P(\mathscr A)$ of the action of $\mathscr P(\cO)$ on $\mathscr A$. The choice of any set of ideals generating $\mathscr P(\cO)$ yields an isogeny graph on the set of vertices $\mathscr P(\mathscr A)$, as described in Example~\ref{example:ppav}.
Now, Theorem~\ref{thm:rapidmixgeneral} provides generating sets $\mathcal S$ with very convenient properties: (i) the corresponding isogeny graph rapidly mixes random walks, and (ii) every edge is an isogeny of small prime degree. In fact, all the occuring isogenies are computable in polynomial time by a recent algorithm of Dudeanu, Jetchev and Robert~\cite{dudeanu-jetchev-robert,Dudeanu16}~(henceforth, the DJR algorithm).

\subsection{Computing isogenies of small degree}\label{subsec:DJR}
More precisely, the DJR algorithm allows to compute any isogeny from $\mathscr J$, defined over $\F_q$ and of odd prime degree $\ell$ (i.e., given a generator of the kernel, it finds an equation of a hyperelliptic curve $\mathscr C'$ such that the target Jacobian is isomorphic to $\Jac(\mathscr C')$) under the following conditions:
\begin{enumerate}
\item\label{condDJB1} $\mathscr J$ has maximal real multiplication, i.e., $\cO_0$ is the maximal order of~$K_0$,
\item\label{condDJB2} the index $[\cO:\Z[\pi,\bar{\pi}]]$ is prime to $2\ell$, and
\item\label{condDJB3} there exists a totally positive element $\beta \in \cO_0$ of norm $\ell$ which annihilates the kernel of the isogeny (the isogeny is called \emph{$\beta$-cyclic}, and the polarisation computed on the target curve depends on the choice of this $\beta$).
\end{enumerate}
The cost of the algorithm is $O(\ell^2)$ operations in $\F_q$, assuming some precomputations of polynomial time in $\log q$ and $\ell$ (see~\cite[Th. 4.8.2]{Dudeanu16}).

Observe that Condition~\eqref{condDJB3} exactly means that the isogeny corresponds to an ideal in the kernel $\mathscr P(\cO)$ of the map $N_{K/K_0} : \Cl(\cO) \rightarrow \Cl^+(\cO_0)$. Therefore this condition is, by construction, satisfied by all the isogenies of the graph.
Also, we can choose the generating set $\mathcal S$ so that it does not contain any ideal of norm dividing the index $[\cO:\Z[\pi,\bar{\pi}]]$, so the isogenies of the graph all satisfy Condition~\eqref{condDJB2} if and only if $[\cO:\Z[\pi,\overline{\pi}]]$ is odd.
Therefore, the conditions
\begin{enumerate}
\item\label{condDJB1} $\mathscr J$ has maximal real multiplication, and
\item\label{condDJB2} the index $[\cO:\Z[\pi,\overline{\pi}]]$ is prime to $2$,
\end{enumerate}
are sufficient for constructing a graph whose edges can all be computed by the DJR algorithm.
Before the work of Dudeanu, Jetchev and Robert, one was only able to compute $(\ell,\ell)$-isogenies \cite{cosset-robert} that were not sufficient to obtain a connected graph.

For the same computational cost, the DJR algorithm can compute the image of a point of order coprime to $2q[\cO:\Z[\pi,\overline{\pi}]]$, given some additional precomputations of polynomial cost in $\mathrm{Disc}(K_0)$.

\subsection{Navigating in the graph with polarizations}\label{subsec:navPolarizations}

The vertices of the graph represent principally \emph{polarizable} (as opposed to \emph{polarized}) abelian surfaces. 
As a consequence, two distinct Jacobians can represent the same vertex if they are isomorphic as abelian varieties, but have non-isomorphic polarizations.
For computations, it is important to be able to determine whether two vertices of the graph are distinct or not, and to this end, the way the vertices are represented is crucial.

As explained in \cite{cosset-robert} and \cite{dudeanu-jetchev-robert}, it is possible to distinguish
between isomorphism classes of Jacobians as principally polarized abelian varieties by simply comparing the Rosenhain invariants\footnote{Since the varieties are absolutely simple, ordinary, and over $\F_q$, two of them are $\F_q$-isomorphic if and only if they are $\bar \F_q$-isomorphic (a consequence of~\cite[Th. 7.2]{Waterhouse1969}; see~\cite[Rem. 3.3]{BJW16}).}. The DJR algorithm computes these explicitly for the target curve of an isogeny. Therefore, if $(\cO_{K_0}^\times)^+/N_{K/K_0}(\cO_{K}^\times)$ is trivial, as discussed in Section~\ref{par:polarizations}, the map $\mathfrak C(\cO) \rightarrow \mathscr P(\cO)$ forgetting the polarization is an isomorphism so the vertices of the graph can simply be represented as Jacobians, or their Rosenhain invariants.

But if $(\cO_{K_0}^\times)^+/N_{K/K_0}(\cO_{K}^\times)$ is of order 2, more work is required. In this case, for any Jacobian $\mathscr J_1$, there exists another Jacobian $\mathscr J_2$ which is isomorphic as a non-polarized abelian variety (and thus represents the same vertex in the graph), but not as a \emph{principally polarized} abelian variety. To solve this issue, one can simply represent the vertices of the graph as pairs of Jacobians, isomorphic as abelian varieties, but with non-isomorphic polarizations. It is still possible to use the DJR algorithm to navigate in this graph. Indeed, let $u \in (\cO_{K_0}^\times)^+$ be a generator of $(\cO_{K_0}^\times)^+ / N_{K /K_0} (\cO_K^\times)$. Starting from $\mathscr J$, given an appropriate kernel, the DJR algorithm chooses a $\beta$ and computes the isogeny as a $\beta$-isogeny, resulting in a target Jacobian $\mathscr J_1$. If $\beta$ is replaced by $u\beta$, the DJR algorithm finds the Jacobian $\mathscr J_2$ which is isomorphic to $\mathscr J_1$ as an abelian variety, but with a different polarization. Therefore the representation of the vertex $\{\mathscr J_1, \mathscr J_2\}$ can be fully computed.

A last point must be addressed: given a Jacobian $\mathscr J$ and a prime $\ell$, the DJR algorithm allows to find isogenies of degree $\ell$ from that Jacobian, but it is unclear a priori which of these isogenies remain within the graph we constructed. Indeed, it could well be that some of these isogenies change the endomorphism order $\cO$. Luckily, this is not a concern if only primes $\ell$ that cannot change the endomorphism order are picked. An isogeny over $\F_q$ of degree $\ell$ can change the order only if $\ell$ divides the index $[\cO_K:\Z[\pi,\bar \pi]]$ (see~\cite[Prop. 3.4]{BJW16}). Therefore, in the generating set $\mathcal S$, we avoid the prime ideals dividing that index.

\subsection{Proof of Theorem \ref{thm:randselfred}}
Let $\mathcal W\subset \mathcal V$ be the subset of all isomorphism classes for which the algorithm $\mathcal A$ solves the DLP.
For any two polarised abelian varieties $\mathscr A$ and $\mathscr B$, write $\mathscr A \sim \mathscr B$ if they are isomorphic as non-polarized abelian varieties. 
Recall that as discussed in Section~\ref{subsec:navPolarizations}, if $\mathcal A$ can solve the DLP on one Jacobian $\mathscr J \in \mathcal W$, then it can solve 
the DLP on the other Jacobians $\mathscr J' \sim \mathscr J$.
Let $V = \mathcal V/\sim$ and $W = \mathcal W/\sim$.
Let $\pi$ be a $q$-Weil number characterising the fixed isogeny class.
From Example~\ref{example:ppav}, the set $V$ is naturally in bijection with $\mathscr P(\mathscr A)$, the orbit for the CM-action of $\mathscr P(\cO)$.
We can therefore apply Theorem~\ref{thm:rapidmixgeneral} on the graph with set of vertices $V$ induced by the set of invertible ideals in $\cO$, coprime to $2[\cO_K:\Z[\pi,\bar \pi]]$, of prime norm bounded by 
$$B_\varepsilon(\cO) = O\left(\left(h_{\cO_0}^+\ln \left (d_KN\mathfrak f [\cO_K:\Z[\pi,\bar \pi]]\right)\right)^{2 + \varepsilon}\right) = O\left(\left(h_{\cO_0}^+\log q\right)^{2 + \varepsilon}\right),$$
where $\mathfrak f$ is the conductor of $\cO$.
 Any path of length at least
$\ln(2| V|/| W|^{1/2}) \leq \ln(2h_\cO)$ starting from any vertex will end in $ W$ with probability between $\mu/2$ and $3\mu/2$. So the strategy to solve DLP on $\mathscr A \in  V$ is to build random paths from $\mathscr A$ in $\mathscr G_B$ of length $\ln(2h_\cO)$ until one of them ends in $ W$, which happens with probability higher than $\mu/2$, so after an expected number of independent trials smaller than $2/\mu$. The length of each path is polynomial in $\ln(h_\cO)$, and the degree of each isogeny on the path is bounded by $B_\varepsilon(\cO)$. So the algorithm computes a polynomial (in $\log q$) number of isogenies, and each of them can be computed in polynomial time (in $\log q$, $h_{\cO_0}^+$ and $\mathrm{Disc}(K_0)$) by the DJR algorithm~\cite{dudeanu-jetchev-robert,Dudeanu16}.

\section{Computing an explicit isogeny between two given Jacobians}\label{sec:galbraith}

Let $\mathscr C$ and $\mathscr C'$ be two hyperelliptic curves of genus 2, defined over the finite field~$\F_q$. Let $\mathscr A = \Jac(\mathscr C)$ and $\mathscr B = \Jac(\mathscr C')$ be their Jacobians. These are principally polarized abelian varieties of dimension 2, so by Tate's isogeny theorem \cite{tate:isogeny}, $\mathscr A$ and $\mathscr B$ are isogenous over $\F_q$ if and only if their  Frobenius polynomials are the same. 
We know how to compute the latter (see \cite{pila:frobenius}, or \cite{GH00} for an efficient algorithm whose running time is $O((\log q)^9)$), and thereby decide whether or not there is an isogeny $\mathscr A \rightarrow \mathscr B$ defined over $\F_q$. 
Yet, once we know that $\mathscr A$ and $\mathscr B$ are isogenous, it is not clear how to explicitly compute an isogeny between them. 
In this section, the expander properties of horizontal isogeny graphs are used to construct and analyse an algorithm similar to Galbraith's algorithm~\cite{Gal99} to build an isogeny between two such varieties having the same endomorphism ring. The contribution of this new algorithm is two-fold. First, the analysis of Galbraith's algorithm relies, in addition to GRH, on some heuristic assumptions on the growing rate of some trees built in the isogeny graph. Using expander properties of these graphs, our analysis relies solely on GRH. Second, while Galbraith's algorithm constructs isogenies between elliptic curves, we provide a more general framework for large families of horizontal isogeny graphs. Precisely, we require

\begin{enumerate}
\item An order $\cO$ of conductor $\mathfrak f$ in a CM-field $K$, and two isogenous abelian varieties $\mathscr A$ and $\mathscr B$ over a finite field $\F_q$ with endomorphism ring $\cO$;
\item A set $\mathcal S$ of ideals in $\cO$ generating a subgroup $H$ of the class group $\Cl(\cO)$, such that the orbits $H(\mathscr A)$ and $H(\mathscr B)$ coincide; 
\item The isogeny graph $\mathscr G$ induced by the action of $H$ on $H(\mathscr A)$ has the rapid mixing property, as described in Theorem~\ref{thm:rapidmixgeneral};
\item The isogenies corresponding to the edges of the graph can be computed in time bounded by some $t_H > 0$.
\end{enumerate}

For elliptic curves, one can choose $H = \Cl(\cO)$, and $\mathcal S$ the set of all ideals of prime norm bounded by a bound $B = O(\log(d_KN\mathfrak f)^{2+\varepsilon}) = O(\left(\log q\right)^{2 + \varepsilon})$. All these isogenies can be computed in time $t_H$ polynomial in $\log q$, and Theorem~\ref{thm:rapidmixgeneral}, or even the less general results of~\cite{jao-miller-venkatesan,JMV09}, shows that $\mathscr G$ has the rapid mixing property. The smaller bound $O(\log(d_KN\mathfrak f)^2)$ was used in Galbraith's approach; the induced graph is then connected, but is not an expander, therefore some additional heuristic assumptions were required for the analysis.

For Jacobians of genus 2 curves, one can choose $H = \mathscr P(\cO)$, and $\mathcal S$ to be a generating set of ideals of prime norms bounded by a bound $B = O(\left(h_{\cO_0}^+\log q)^{2 + \varepsilon}\right)$, where $\cO_{0} = \cO \cap K_0$.
As seen in Section~\ref{subsec:DJR}, the corresponding isogenies can then be computed using the DJR algorithm when $\cO_0$ is maximal and $[\cO:\Z[\pi,\bar\pi]]$ is odd.

Write $h = |H|$. The idea is to find $h^{1/2}$ varieties ``close" to $\mathscr  A$ (in the sense that we know a path of polynomial length from these to $\mathscr  A$), and then to build paths out of $\mathscr  B$ until one of the neighbors of $\mathscr  A$ is reached. In practice one could simply use the same tree-growing strategy as Galbraith~\cite{Gal99}, but the analysis of our algorithm requires the various random paths to be independent in order to use the expanding properties (and this independence misses in the ``tree" approach). The algorithm goes as follows, presented in the most general setting.
\begin{enumerate}
\item[\textbf{Step 1}] Build independent random paths in $\mathscr G$ of length $\ln(2h)$ from $\mathscr  A$ until $h^{1/2}$ vertices are reached. Those are the \emph{neighbors} of $\mathscr  A$.
\item[\textbf{Step 2}] Build independent random paths of length $\ln(2h)$ from $\mathscr  B$ until a neighbor of $\mathscr  A$ is reached. There is now a short path between $\mathscr  A$ and $\mathscr B$.
\end{enumerate}

Now, let us prove that the number of paths considered at each step is on average $O(h^{1/2})$. Let $Y$ be a subset of the vertices of $\mathscr G$, smaller than $2h/3$. By a \emph{trial}, we mean the computation of a random path of length $\ln(2h)$ from of $A$, and a trial is a \emph{success} if the path ends out of $Y$. Let us estimate the number $N_Y$ of independent trials we need to obtain a success,
\begin{equation*}
\mathbf E[N_Y] = \sum_{i=1}^{\infty} i\Pr [i-1\text{ failures and 1 success}]
 \leq \sum_{i=1}^{\infty} i \left(\frac{3|Y|}{2h}\right)^{i-1},
\end{equation*}
and from the generating function $(1-x)^{-2} = \sum_{i=0}^{\infty}(i+1)x^i$, we obtain the inequality\begin{alignat*}{1}
\mathbf E[N_Y] & \leq \left(1-\frac{3|Y|}{2h}\right)^{-2} = \frac{4h^2}{(2h - 3|Y|)^2}.
\end{alignat*}
Now consider the experiment consisting in a sequence of independent trials, and let $Y_n$ be the first $n$ distinct points obtained from the first experiments. The number $M_n$ of trials required to find $n$ distinct points can be estimated as
\begin{alignat*}{1}
\mathbf E[M_n] = \sum_{i=1}^{n-1}\mathbf E[N_{Y_{i}}]\leq\sum_{i=1}^{n-1}\frac{4h^2}{(2h - 3i)^2} \leq \frac{4nh^2}{(2h - 3n)^2}.
\end{alignat*}
In particular, to find $h^{1/2}$ neighbors of $\mathscr A$, the expected number of trials 
$
{\mathbf E[M_{h^{1/2}}] 
}
$
is at most $4h^{1/2}$,
assuming that $h$ is at least 9. Of course, in practice, we expect to need much less trials since we count here only the end point of each path. This proves that the expected number of paths we have to compute in Step 1 is $O\left(h^{1/2}\right)$.

The expected number of paths considered in Step 2 can be found to be $O\left(h^{1/2}\right)$ in a similar fashion.
In total, we build $O\left(h^{1/2}\right)$ paths of length $O(\ln h)$. So the algorithm needs to compute $O\left(h^{1/2}\ln h\right)$ isogenies, each of them being computable in time $t_H$, and finds a path of length $O(\ln h)$ between $\mathscr  A$ and $\mathscr  B$.

\section*{Acknowledgements}
We thank Emmanuel Kowalski, Philippe Michel, Ken Ribet and Damien Robert for useful conversations.
The first author was supported by the Swiss National Science Foundation. The second author was supported by the Swiss National Science Foundation under grant number 200021-156420.

\bibliographystyle{amsalpha}
\def\cprime{$'$}
\providecommand{\bysame}{\leavevmode\hbox to3em{\hrulefill}\thinspace}
\providecommand{\MR}{\relax\ifhmode\unskip\space\fi MR }
\providecommand{\MRhref}[2]{%
  \href{http://www.ams.org/mathscinet-getitem?mr=#1}{#2}
}
\providecommand{\href}[2]{#2}

\end{document}